\documentclass[12pt]{article}
\usepackage[T1]{fontenc}
\usepackage[utf8]{inputenc}
\usepackage[a4paper, total={6.5in, 9in}]{geometry}
\usepackage{amsmath,amssymb,amsthm,tikz}
\usepackage{graphicx,xcolor}
\usepackage[hang]{subfigure}
\usetikzlibrary{positioning}
\usetikzlibrary{calc}
\usetikzlibrary{decorations.markings}
\theoremstyle{definition}
\newtheorem{definition}{Definition}[section]
\newtheorem{theorem}{Theorem}[section]

\newtheorem{lemma}[theorem]{Lemma}
\usepackage{mathtools}
\newcommand\vertarrowbox[3][6ex]
{
\begin{array}[t]{@{}c@{}} #2 \\
\left\uparrow\vcenter{\hrule height #1}\right.\kern-\nulldelimiterspace\\
\makebox[0pt]{\scriptsize#3}
\end{array}
}
\title{Virtually special non-finitely presented groups via linear characters}
\author{Vladimir Vankov\footnote{School of Mathematical Sciences, University of Southampton, v.vankov@soton.ac.uk}}
\date{\today}
\usepackage{hyperref}
\hypersetup
{
colorlinks=true,
linkcolor=black,
citecolor=black
}

\begin{document}

\tikzset{middlearrow/.style=
{
        decoration=
        {markings,
            mark= at position 0.53 with {\arrow{#1}} ,
        },
        postaction={decorate}
    }
}
 
\maketitle

\begin{abstract}
We present a new method for showing that groups are virtually special. This is done by considering finite quotients and linear characters. We use this to show that an infinite family of groups, related to Bestvina-Brady groups and branching, provides new examples of virtually special groups outside of a hyperbolic context.
\end{abstract}
 
\section{Introduction}
\label{sec:1}

Given a simplicial graph $\Gamma$, the associated \textit{right-angled Artin group} (RAAG) is the group $A_\Gamma$ with generators the vertices of $\Gamma$, and relations being commutators between vertices joined by an edge. Complete graphs give free abelian groups and empty graphs give free groups. Subgroups of RAAGs have many rich properties, such as being residually finite.

Consider the \textit{Bestvina-Brady group} from \cite{bestvinabrady}, corresponding to the minimal flag triangulation $L$ of the circle. The associated group $A_L$ splits as $F_2\times F_2$ in this case, and we can get a presentation:
$$
G=\left<x_1,x_2,x_3,x_4\ \middle|\ x_1^i x_2^i x_3^i x_4^i\text{ for }i\in\mathbb{Z}\right>.
$$
The group $G$ is defined as the kernel of a map from $A_L$ to $\mathbb{Z}$ (more on this in Section \ref{sec:3}), therefore is naturally a subgroup of a RAAG.

We ask what happens if we vary the number of generators, and replace some of the relations by their powers. The resulting groups would have torsion, so cannot be subgroups of RAAGs. Nonetheless, we could look at torsion-free finite-index subgroups, and see if we can embed those into RAAGs. Certain properties of subgroups of RAAGs, such as residual finiteness, pass between finite-index subgroups.

\begin{theorem}
For each integer $m\geqslant4$ and prime $k\geqslant2$, the group
$$
G_m^k:=\left<x_1,x_2,\dots,x_m\ \middle|
\begin{aligned}
&\ \ x_1^i x_2^i\cdots x_m^i\\
&\left(x_1^i x_2^i\cdots x_m^i\right)^k
\end{aligned}
\ 
\begin{aligned}
&\text{ for }i\in k\mathbb{Z}\\
&\text{ for }i\in \mathbb{Z}\backslash k\mathbb{Z}
\end{aligned}\right>
$$

is virtually special. In particular, it is residually finite.
\label{thm:main}
\end{theorem}

Special cube complexes were introduced by Haglund and Wise in \cite{special}. They are nonpositively curved cube complexes whose fundamental groups embed into RAAGs. We say that a group is \textit{virtually special} if it has the fundamental group of a special cube complex as a finite-index subgroup.

Bestvina-Brady groups have classifying spaces which are special cube complexes. The groups in Theorem \ref{thm:main} arise from branched covers of classifying spaces of Bestvina-Brady groups corresponding to different flag triangulations of the circle. In particular, the group $G_m^k$ corresponds to a generalised Bestvina-Brady group $G_L^M(S)$ by taking $M\to L$ to be the $k$-regular covering of the regular $m$-gon by the regular $km$-gon (thought of as flag triangulations of the circle), with branching set $S=k\mathbb{Z}$ (more on this in Section \ref{sec:3}). This kind of construction was first introduced in \cite{uncountably}.

The idea of the proof  of Theorem \ref{thm:main} first came from the author representing edges in complexes by matrices, using representations of finite quotients and height functions to work in some $GL_n(V)\oplus\mathbb{Z}$. This allows us to turn geometric information about hyperplanes into algebraic statements. After noticing that considering determinants of these matrices can be useful, linear characters entered the picture.

A lot of progress concerning special cube complexes has been surrounding hyperbolic groups (such as Agol's theorem \cite{agol}), however in our case the virtually special groups are not finitely presented (see Section \ref{sec:7}), hence not hyperbolic. Furthermore, we have explicit presentations for our groups.

The paper is structured as follows. Sections \ref{sec:2} and \ref{sec:3} contain the relevant background on special cube complexes and Bestvina-Brady groups, respectively. Section \ref{sec:4} describes the action of the groups $G_m^k$ on the associated CAT(0) complexes $X_m^k$ in terms of fundamental domains (as well as the \hyperlink{apx}{appendix}). Section \ref{sec:5} shows that the groups are virtually torsion-free and introduces specific quotients. The quotient complex is considered, and the structure of hyperplanes is examined. Section \ref{sec:6} showcases the method of utilising linear characters to prove that the quotient complex is special, and proves Theorem \ref{thm:main}. Section \ref{sec:7} distinguishes the groups, examines particular infinite behaviour of the hyperplanes and considers further generalisations.

I would like to thank Ian Leary, my PhD supervisor, for the support and great conversations full of helpful suggestions.

\section{Background on Special Cube Complexes}
\label{sec:2}

We will only need to consider 2-dimensional square complexes in this paper. Let $I=[0,1]$ be the unit interval. An \textit{n-cube} for $n>0$ is a copy of $I^n$ and a 0-cube is just a vertex. A \textit{cube complex} is a cell complex where every cell is some $n$-cube, with the attaching maps being combinatorial on cubes. This means that the attaching maps must send $n$-cubes isometrically to $n$-cubes when we consider the boundaries of cubes as being the union of lower-dimensional cubes. In our case, squares will be glued together by sending vertices to vertices and edges to edges.

The \textit{link} of a vertex $v$ is a simplicial complex whose vertices correspond to ends of 1-cubes attached to $v$, and these are joined by an $n$-simplex for each corresponding $(n+1)$-cube having a corner at $v$. In our case, links will be simplicial graphs.

A \textit{flag complex} is a simplicial complex where any finite collection of vertices that are pairwise joined by a 1-simplex span a simplex. We say that a cube complex is \textit{nonpositively curved} if the link of every vertex is a flag complex. In our case, this means that every link is triangle-free. We refer to simply connected nonpositively curved cube complexes as $\text{CAT}(0)$ cube complexes.

Let $X$ be a nonpositively curved  cube complex. `Square' will refer to a 2-cube in $X$, and `edge' will refer to a 1-cube. For edges $u$, $v$ which are opposite each other in some square of $X$, we write $u\sim v$ and say that $u$ and $v$ are elementary parallel. This induces an equivalence relation on the edges. By abuse of notation, we write $u\sim w$ if edges $u,w$ lie in the same equivalence class and say they are parallel. We denote the equivalence class of $u$ by $[u]$, and we call this a hyperplane. If we induce an orientation on edges in such a way that elementary parallelism keeps track of orientation, we say that a hyperplane $[u]$ is not two-sided if $u\sim-u$, and two-sided otherwise. We also write $[u]\sim[v]$ if $u,v$ lie in the same hyperplane.

\begin{definition}[Hyperplane interactions, \cite{special}]
If two edges $u,v$ are adjacent in a square (intersect at a corner of the square), we write $u\perp v$. We write $[u]\perp[v]$ and say that hyperplanes $[u],[v]$ cross if there exist edges $u',v'$ such that $u'\perp v'$ and $u'\sim u$, $v'\sim v$. If two edges $w,x$ share a vertex (intersect at a 0-cell), but there does not exist a square that contains both of them where they are adjacent (i.e. $w\not\perp x$), we write $w\circlearrowright x$. We write $[w]\circlearrowright[x]$ and say that hyperplanes $[w],[x]$ osculate if there exist edges $w',x'$ such that $w'\circlearrowright x'$ and $x'\sim x$, $w'\sim w$.
\end{definition}

One can define a special cube complex in terms of avoiding certain configurations of hyperplane interactions. Note that while $\circlearrowright$ and $\perp$ are relations, only $\sim$ is an equivalence relation.

\begin{definition}[Special Cube Complex, \cite{special}]
\label{pre:1}
If $X$ is such that for all edges $u$, $[u]$ is two-sided, and for any pair of (not necessarily distinct) edges $u,v$, we have at most one of the relations $\sim$, $\perp$, $\circlearrowright$ holding between $[u]$ and $[v]$, then we say that $X$ is a special cube complex.
\end{definition}

If $\sim$ and $\perp$ hold, then a hyperplane crosses itself, so is not an embedded hyperplane. If $\sim$ and $\circlearrowright$ hold, then a hyperplane self-osculates. If both $\perp$ and $\circlearrowright$ hold between a pair of hyperplanes, they inter-osculate. There are essentially 4 components to Definition \ref{pre:1}, we will refer to them in Section \ref{sec:6}.

\begin{theorem}[Haglund and Wise, \cite{special}]
\label{pre:wise}
If $X$ is a special cube complex, then $\pi_1(X)$ embeds into a Right-Angled Artin Group. In particular, if $X$ contains finitely many hyperplanes, then $\pi_1(X)$ embeds into a finitely generated Right-Angled Artin Group.
\end{theorem}

Having finitely many hyperplanes is important, as a finitely-generated Right Angled Artin Group in particular embeds into $SL_n(\mathbb{Z})$ for some finite $n$.

\section{Background on Bestvina-Brady Groups}
\label{sec:3}

Bestvina-Brady groups (introduced in \cite{bestvinabrady}) are by definition normal subgroups of Right-Angled Artin Groups (RAAGs), hence they are linear over $\mathbb{Z}$ and enjoy properties such as being residually finite. If a space has fundamental group $G$ and has contractible universal cover, we say that space is a \textit{classifying space} for $G$.

\begin{definition}[Salvetti complex]
Let $A_\Gamma$ be a RAAG defined by a graph $\Gamma$. Then $A_\Gamma$ has a natural classifying space, called the Salvetti complex, formed from one vertex, one loop for each Artin generator, and an $n$-torus for each $n$-clique in $\Gamma$, glued appropriately.
\end{definition}

\hypertarget{sdbl}{The} link of the vertex in the Salvetti complex corresponding to $A_\Gamma$ is a ``spherical double'' of the flag complex with 1-skeleton $\Gamma$. This is formed by taking each vertex $v$ and replacing it with two vertices $v^{+},v^{-}$, such that a set of vertices form a simplex if the corresponding vertices formed a simplex when forgetting about the superscripts. Note that this remains a flag complex if the original complex was flag. We will focus on the case when these links are finite complexes.

\begin{definition}[Bestvina-Brady Group]
Let $L$ be a connected finite flag complex. Define $BB_L$ to be the kernel of the homomorphism from the RAAG $A_L$ (associated with the 1-skeleton of $L$) to $\mathbb{Z}$, which sends every Artin generator to 1 in $\mathbb{Z}$.
\end{definition}

The finiteness properties of $BB_L$ are controlled by the choice of complex $L$. Note that there is a natural correspondence between flag complexes $L$ and their 1-skeletons $\Gamma$, hence we can refer to a group $A_L$. One can think of the naturally associated classifying space $\mathbb{B}_L$ either as the quotient of the universal cover of the Salvetti complex of $A_L$ by $BB_L$, or as a $\mathbb{Z}$-cover of this Salvetti complex (which gives us a natural height function $f$). Notice that the ascending and descending links (which are the induced subcomplexes of the link on edges pointing either up or down, respectively) of vertices in $\mathbb{B}_L$ are all isomorphic to $L$. Since there are only countably many finite connected flag complexes, there are at most countably many groups $BB_L$.

This construction was generalised by Leary in \cite{uncountably}, to give uncountably many more such groups for each finite connected flag complex.

\begin{definition}[Generalised Bestvina-Brady Group]
\label{pre:2}
Let $M\to L$ be a regular cover. Following the start of section 21 in \cite{uncountably}, define $G_L^M(S)$ to be the group of deck transformations of the branched cover $X_L^M(S)$ of the classifying space $\mathbb{B}_L$ of $BB_L$. The branching occurs at vertices with heights not in $S$, when the $BB_L$-orbits of vertices in the universal cover of $\mathbb{B}_L$ are naturally labelled with the integers (consider the height function $f$, mentioned above). The branching is such that the ascending and descending links of branching vertices are isomorphic to $M$.
\end{definition}

The existence of such a branched cover is Theorem 9.1 in \cite{uncountably}. For the rest of the paper, let $m>3$ be an integer and $k$ a prime number. We will focus on the case when $L$ is the $m$-vertex flag triangulation of the circle, and $M$ is the $mk$-vertex flag triangulation of the circle. In Section \ref{sec:4}, we will give specific names to vertices, edges and squares in $X_L^M$ in this case.

When the group does not coincide with a Bestvina-Brady group, it does not act specially on its associated $\text{CAT}\left(0\right)$ complex (using language from \cite{coxeter}, see Definition 3.4 there), because of the action of the point stabilisers at the branching vertices. Indeed, those where this complex is locally finite (call these groups ``of finite ramification'') contain torsion, hence cannot be subgroups of RAAGs.

\begin{definition}[Finite Ramification]
When the group of deck transformations of $M\to L$ is finite, we say that $G_L^M(S)$ is of finite ramification (i.e. $[\pi_1(L):\pi_1(M)]<\infty$). To avoid this group coinciding with a Bestvina-Brady group, we insist that $S\neq\mathbb{Z}$ and $\pi_1(L)/\pi_1(M)$ is not the trivial group.
\end{definition}

In particular, this means that we study the case when $L$ is not simply connected and $M\neq L$. Note that when $M$ is the universal cover of $L$ and $\pi_1(L)$ is finite, we get finite ramification of the group in the main statement of \cite{uncountably} without any modification, as long as $S\neq\mathbb{Z}$. We have decided to exclude Bestvina-Brady groups from this definition, because they are special by definition, and we are interested in new virtually special groups.

\hypertarget{theset}{Let} $\mathfrak{R}$ be the set of generalised Bestvina-Brady groups of finite ramification which are virtually special, up to isomorphism. We conclude this section by showing that $\mathfrak{R}$ is countably infinite.

\begin{theorem}
\label{crit}
Assume $G_{L}^M(S)$ is of finite ramification, then it is virtually torsion-free only if $S$ is periodic.
\end{theorem}

\begin{proof}
Assume that $G_{L}^M(S)$ is virtually torsion-free. Then there exists $H\dot{<}G_{L}^M(S)$ which is torsion-free, and in particular $H\not=G_{L}^M(S)$ because $S\not=\mathbb{Z}$, as we are assuming the group is not a Bestvina-Brady group. This implies that there is torsion in $G_{L}^M(S)$, coming from finite point stabilisers, as the group of deck transformations $M\to L$ is finite and non-trivial. Let $X=G_{L}^M(S)/H$ denote the set of left cosets of $H$ in $G_{L}^M(S)$. There is a left $G_{L}^M(S)$-action on $X$:
$$
G_{L}^M(S)\times X\to X\;,\;\;g'\cdot(gH)=(g'g)H\;,\;\;\forall\;g',g\in G_{L}^M(S).
$$
Because $H$ is of finite index in $G_{L}^M(S)$, we have $r=|X|<\infty$. The action gives a homomorphism
$$
\phi:G_{L}^M(S)\to Sym_X\cong S_r.
$$
Since $gH=H\implies g\in H$, we get $\ker{\phi}\leqslant H$. 

Let $p$ be a prime number dividing the order of the group of deck transformations $M\to L$ and let $\gamma$ be a loop in $L$ whose representative in this group has order $p$. Choose an edge loop $a_1,\dots,a_m$ in $L$ corresponding to $\gamma$.

Utilising the presentation of Definition 1.1 in \cite{uncountably}, we get group generators from edges, so consider $\phi(a_1),\ \dots,\ \phi(a_m)$. These are finite group elements in the finite group $S_r$. As $\phi$ is a homomorphism, denoting the order of an element $g$ by $o(g)$,we get the infinite sequence
$$
J=\ldots,\ o\left(\phi(a_1)^{-1}\cdots\phi(a_m)^{-1}\right),\ o\left(\phi(a_1)^0\cdots\phi(a_m)^0\right)=1,
$$
$$
o\left(\phi(a_1)\cdots\phi(a_m)\right),\ o\left(\phi(a_1)^2\cdots\phi(a_m)^2\right),\dots
$$
This is now a periodic sequence with each term in $\{1,p\}$, with some finite period. Now $H$ being torsion-free implies that $\ker{\phi}$ is also torsion-free. For $i\not\in S$, $a_1^i \cdots a_m^i$ is a torsion element, and hence cannot belong to $\ker{\phi}$, therefore $o\left(\phi(a_1)^i\cdots\phi(a_m)^i\right)=p$. But for $i\in S$, we have $a^i\cdots a_m^i=1$, and since $\phi$ is a homomorphism, $o\left(\phi(a_1)^i\cdots\phi(a_m)^i\right)=1$ too. This means that if $S$ is not periodic, then $J$ is not periodic. Hence
$$
G_{L}^M(S)\ \text{virtually torsion-free}\implies J\ \text{periodic}\implies S\ \text{periodic}.
$$
\end{proof}

Since there are countably many periodic subsets of $\mathbb{Z}$, and countably many finite connected flag complexes, there are at most countably many generalised Bestvina-Brady groups of finite ramification which are virtually torsion-free.  Groups which are virtually special must be virtually torsion-free. Hence $\mathfrak{R}$ is at most countable. It is not finite because Theorem \ref{thm:main} provides infinitely many examples (see Section \ref{sec:7} for why they are pairwise non-isomorphic).

\section{Fundamental Domain}
\label{sec:4}

We apply Definition \ref{pre:2} to $L$ being an $m$-vertex triangulation of the circle and $M$ being an $mk$-triangulation of the circle. We refer to the resulting group as $G_m^k$, and to the branched cover as $X_m^k$. Using the presentation in \cite{uncountably}, we get generators $x_1,\dots,x_m$ and relations as stated in Theorem \ref{thm:main}.

The group $G_m^k$ acts (on the left) on the square complex $X_m^k$ which admits a height function $f:X_m^k\to\mathbb{R}$ such that vertices of the complex lie at integer heights. There is one orbit of vertices at each height $i$, with a distinguished base vertex $X^i$, which has stabiliser $\left<x_1^ix_2^i\cdots x_m^i\right>$. Other vertices will have stabilisers being the appropriate conjugate of this.

Every edge joins two vertices of heights differing by 1. The edges are labelled by the heights of the top vertex. There are $m$ free orbits of edges at each height, with distinguished orbit representatives $\mathbf{u}^i_j$ at height $i$ for $1\leqslant j\leqslant m$, such that $\mathbf{u}_1^i$ joins $X^{i-1}$ to $X^i$. Every square has a top vertex of height $i+1$, a bottom vertex of height $i-1$ and two vertices of height $i$, with such a square being labelled with height $i$. Denote by $\text{Orb}\left(\mathbf{u}\right)$ the orbit of an edge. There are $m$ free orbits of squares at each height, with distinguished orbit representatives $\underline{\mathbf{s}}^i_j$ at height $i$ for $1\leqslant j\leqslant m$. They are chosen such that $\underline{\mathbf{s}}^i_j$ contains edges in $\text{Orb}\left(\mathbf{u}_j^{i+1}\right)$, $\text{Orb}\left(\mathbf{u}_{j+1}^{i+1}\right)$, $\text{Orb}\left(\mathbf{u}_j^{i}\right)$, $\text{Orb}\left(\mathbf{u}_{j+1}^{i}\right)$, in that order, having picked an appropriate direction to read around the square (using cyclic indexing, which will also be used later on, which means that $j+1$ denotes 1 for $j=m$, for example). In order to complete the convention for labelling the edges and squares, we use Figure \ref{fig:l1} to order the edges and squares. Note that the labels on the edges in the link refer to which squares in the complex $X_m^k$ contribute towards this.

\begin{figure}[h!]
\centering
\begin{tikzpicture}
\coordinate (left) at (0,0);
\coordinate (mid1) at (2.5cm,0);
\coordinate (mid2) at (5cm,0);
\coordinate (right) at (7.6cm,0);
\coordinate (rightest) at (10.7cm,0);
\draw[very thick] (left) -- node[below] {$\underline{\mathbf{s}}_1^i$} (mid1);
\draw[very thick] (rightest) -- node[below left] {$\underline{\mathbf{s}}_m^i$} (right);
\draw[very thick] (mid2) -- node[below] {$\underline{\mathbf{s}}_{m-1}^i$} (right);
\draw[very thick,dashed] (mid1) -- (mid2);
\draw (left) node [draw, fill=gray!12,rounded corners] {$\mathbf{u}_1^i$};
\draw (mid1) node [draw,fill=gray!12,rounded corners] {$\mathbf{u}_2^i$};
\draw (mid2) node [draw,fill=gray!12,rounded corners] {$\mathbf{u}_{m-1}^i$};
\draw (right) node [draw,fill=gray!12,rounded corners] {$\mathbf{u}_{m}^i$};
\draw (rightest) node [draw,fill=gray!12,rounded corners] {$x_1^i\cdots x_m^i\cdot\mathbf{u}_{1}^i$};
\end{tikzpicture}
\caption{Part of the descending link of $X^i$.}
\label{fig:l1}
\end{figure}
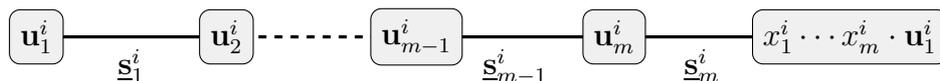

Note further that in Figure \ref{fig:l1}, when $i\in S$, the element $x_1^i\cdots x_m^i$ is the identity, so this forms a complete loop, isomorphic to $L$.

For a vertex of non-0 height, define the shadow to be the set of vertices of height 0 which can be reached by moving along connected edges, such that the height always strictly decreases/increases to 0. Vertex $X^i$, for positive $i$, has the following vertices on the boundary of its shadow: (see Lemma 14.3 in \cite{uncountably})

\begin{align*}
 & & & & & & X^0 &,\\
 x_1\cdot X^0&,& x_1^2\cdot X^0&,& \dots&\ ,& x_1^i\cdot X^0&,\\
  x_1^i x_2\cdot X^0&,& x_1^i x_2^2\cdot X^0&,& \dots&\ ,& x_1^i x_2^i\cdot X^0&,\\
  & & \vdots& & & & & \\
  x_1^i x_2^i \cdots x_{m-1}^i x_m\cdot X^0&,& x_1^i x_2^i \cdots x_{m-1}^i x_m^2\cdot X^0&,& \dots&\ ,& x_1^i x_2^i\cdots x_m^i\cdot X^0&
\end{align*}

\bigskip

in that order (when reading in an appropriate direction). Note that when $i\in k\mathbb{Z}$, then $x_1^i x_2^i \cdots x_m^i$ is the identity element, hence $x_1^i x_2^i\cdots x_m^i\cdot X^0=X^0$ and the elements above are a complete loop. Vertex $X^i$, for negative $i$, has the following vertices on the boundary of its shadow:

\begin{align*}
 & & & & & & X^0 &,\\
 x_1^{-1}\cdot X^0&,& x_1^{-2}\cdot X^0&,& \dots&\ ,& x_1^i\cdot X^0&,\\
  x_1^i x_2^{-1}\cdot X^0&,& x_1^i x_2^{-2}\cdot X^0&,& \dots&\ ,& x_1^i x_2^i\cdot X^0&,\\
  & & \vdots& & & & & \\
  x_1^i x_2^i \cdots x_{m-1}^i x_m^{-1}\cdot X^0&,& x_1^i x_2^i \cdots x_{m-1}^i x_m^{-2}\cdot X^0&,& \dots&\ ,& x_1^i x_2^i\cdots x_m^i\cdot X^0&.
\end{align*}

These can be thought of as the boundaries of the base of a pyramid with base at height 0 and apex vertex $X^i$. The faces of the pyramid are parts of embedded planes in $X_m^k$, consisting of two types of edges each. When $i\in S$, this pyramid will have $m$ faces and it will have $mk$ faces otherwise.

\begin{lemma}
\label{funlemma}

Given the above notational conventions, the complex $X^k_m$ and the action of $G_m^k$ on it can be fully described in Figure \ref{fig:bigfun}.

\begin{figure}[h]
\makebox[\linewidth][c]{%
\centering
\subfigure[$1\leqslant j<m$]{
\begin{tikzpicture}
\coordinate (bottom) at (0,-2cm);
\coordinate (top) at (0,2cm);
\coordinate (left) at (-2cm,0);
\coordinate (right) at (2cm,0);
\fill[gray!20] (bottom) -- (left) -- (top) -- (right) -- cycle;
\draw[very thick] (bottom) -- node[below left] {$\alpha(i,j-1)\cdot\mathbf{u}_{j+1}^i$} (left);
\draw[very thick] (bottom) -- node[below right] {$\alpha(i,j)\cdot \mathbf{u}_j^i$} (right);
\draw[very thick] (right) -- node[above right] {$\mathbf{u}_{j+1}^{i+1}$} (top);
\draw[very thick] (left) -- node[above left] {$\mathbf{u}_j^{i+1}$} (top);
\draw[very thick,dashed] (left) -- node[above] {$\underline{\mathbf{s}}^i_j$} (right);
\draw (bottom) node [draw, fill=gray!12,rounded corners] {$\beta(i,j)\cdot X^{i-1}$};
\draw (left) node [draw, fill=gray!12,rounded corners] {$\alpha(i,j-1)\cdot X^i$};
\draw (right) node [draw, fill=gray!12,rounded corners] {$\alpha(i,j)\cdot X^i$};
\draw (top) node [draw, fill=gray!12,rounded corners] {$X^{i+1}$};
\end{tikzpicture}
}
\subfigure[$j=m$]{
\begin{tikzpicture}
\coordinate (bottom) at (0,-2cm);
\coordinate (top) at (0,2cm);
\coordinate (left) at (-2cm,0);
\coordinate (right) at (2cm,0);
\fill[gray!20] (bottom) -- (left) -- (top) -- (right) -- cycle;
\draw[very thick] (bottom) -- node[below left] {$\gamma(i,m)x_m^{-1}\cdot \mathbf{u}_{1}^i$} (left);
\draw[very thick] (bottom) -- node[below right] {$\alpha(i,m)\cdot \mathbf{u}_m^i$} (right);
\draw[very thick] (right) -- node[above right] {$\gamma(i,m)\cdot \mathbf{u}_{1}^{i+1}$} (top);
\draw[very thick] (left) -- node[above left] {$\mathbf{u}_m^{i+1}$} (top);
\draw[very thick,dashed] (left) -- node[above] {$\underline{\mathbf{s}}^i_m$} (right);
\draw (bottom) node [draw, fill=gray!12,rounded corners] {$\beta(i,m)\cdot X^{i-1}$};
\draw (left) node [draw, fill=gray!12,rounded corners] {$\alpha(i,m-1)\cdot X^i$};
\draw (right) node [draw, fill=gray!12,rounded corners] {$\alpha(i,m)\cdot X^i$};
\draw (top) node [draw, fill=gray!12,rounded corners] {$X^{i+1}$};
\end{tikzpicture}
}
}
\caption{$G_m^k$-orbit representatives of 2-cells in $X_m^k$ in layer $i$.}
\label{fig:bigfun}
\end{figure}
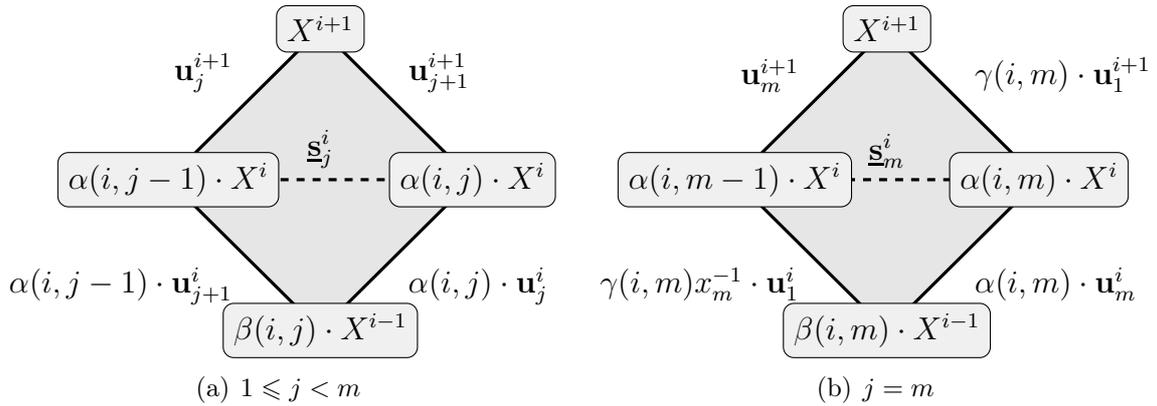

\end{lemma}

In Figure \ref{fig:bigfun}, the following shorthand is used: (note that here, $x_i$ for $i<1$ should be thought of as the identity element, for example $\alpha(i,0)$ is just the identity element)
$$
\alpha(i,j):= x_1^{i+1}\cdots x_{j-1}^{i+1}\ x_j\ x_{j-1}^{-i}\cdots x_1^{-i}
$$
$$
\beta(i,j):= x_1^{i+1}\cdots x_{j-1}^{i+1}\ x_j\ x_{j-1}^{1-i}\cdots x_1^{1-i}
$$
$$
\gamma(i,j):= x_1^{i+1}\cdots x_j^{i+1}
$$

The proof of this involves unwrapping the notation by induction and can be found in the \hyperlink{apx}{appendix}. Much of the proof of Theorem \ref{thm:main} will involve translating around patches of the fundamental domain to understand the structure of hyperplanes, using the fact that edges or squares of a fixed height and type are in free orbit.

\section{Hyperplane Stabilisers in a Quotient}
\label{sec:5}

We show that $G_m^k$ is virtually torsion-free by exhibiting an explicit surjection onto a finite group with torsion-free kernel as follows:

\begin{align*}
\phi_m^k:G_m^k&\to \bar{G}_m^k\cong \overbrace{ C_k \times \cdots \times C_k}^{m\ \text{copies}}\\
x_i&\mapsto (0,\dots,0,\vertarrowbox[3ex]{\sigma_i}{$i$th position},0,\dots,0)
\end{align*}

This is a homomorphism since each $\sigma_i$ has order $k$. The only torsion elements come from point stabilisers of the action on $X_m^k$, which are conjugates of the torsion elements in the presentation (or their powers). Since these do not map to the identity (neither do their powers, as $k$ is prime) in $\bar{G}^k_m$, we get that $\ker{\phi^k_m}$ is a torsion-free subgroup of $G_m^k$ of index $mk$. Note that $\bar{G}_m^k$ is abelian. We could have used a smaller target group to get a torsion-free kernel, however this larger group will be useful for defining linear characters later. If $k$ were not prime, this would be false, as then an intermediate power of a torsion element (which is a torsion element itself) would still lie in the kernel, due to some factorisation of $k$.

We can now turn our attention to the quotient complex $\bar{X}_m^k:=X_m^k/\ker{\phi_m^k}$. We denote images by placing a bar over the notation. From Lemma \ref{funlemma}, we obtain the new fundamental domain in Figure \ref{fig:smallfun}.

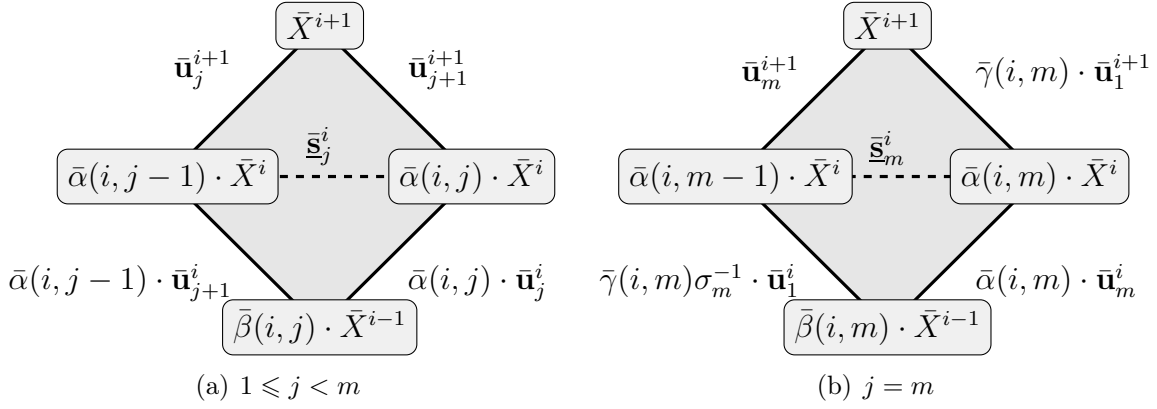
\begin{figure}[h!]
\makebox[\linewidth][c]{%
\centering
\subfigure[$1\leqslant j<m$]{
\begin{tikzpicture}
\coordinate (bottom) at (0,-2cm);
\coordinate (top) at (0,2cm);
\coordinate (left) at (-2cm,0);
\coordinate (right) at (2cm,0);
\fill[gray!20] (bottom) -- (left) -- (top) -- (right) -- cycle;
\draw[very thick] (bottom) -- node[below left] {$\bar{\alpha}(i,j-1)\cdot\bar{\mathbf{u}}_{j+1}^i$} (left);
\draw[very thick] (bottom) -- node[below right] {$\bar{\alpha}(i,j)\cdot \bar{\mathbf{u}}_j^i$} (right);
\draw[very thick] (right) -- node[above right] {$\bar{\mathbf{u}}_{j+1}^{i+1}$} (top);
\draw[very thick] (left) -- node[above left] {$\bar{\mathbf{u}}_j^{i+1}$} (top);
\draw[very thick,dashed] (left) -- node[above] {$\underline{\bar{\mathbf{s}}}^i_j$} (right);
\draw (bottom) node [draw, fill=gray!12,rounded corners] {$\bar{\beta}(i,j)\cdot \bar{X}^{i-1}$};
\draw (left) node [draw, fill=gray!12,rounded corners] {$\bar{\alpha}(i,j-1)\cdot \bar{X}^i$};
\draw (right) node [draw, fill=gray!12,rounded corners] {$\bar{\alpha}(i,j)\cdot \bar{X}^i$};
\draw (top) node [draw, fill=gray!12,rounded corners] {$\bar{X}^{i+1}$};
\end{tikzpicture}
}
\subfigure[$j=m$]{
\begin{tikzpicture}
\coordinate (bottom) at (0,-2cm);
\coordinate (top) at (0,2cm);
\coordinate (left) at (-2cm,0);
\coordinate (right) at (2cm,0);
\fill[gray!20] (bottom) -- (left) -- (top) -- (right) -- cycle;
\draw[very thick] (bottom) -- node[below left] {$\bar{\gamma}(i,m)\sigma_m^{-1}\cdot \bar{\mathbf{u}}_{1}^i$} (left);
\draw[very thick] (bottom) -- node[below right] {$\bar{\alpha}(i,m)\cdot \bar{\mathbf{u}}_m^i$} (right);
\draw[very thick] (right) -- node[above right] {$\bar{\gamma}(i,m)\cdot \bar{\mathbf{u}}_{1}^{i+1}$} (top);
\draw[very thick] (left) -- node[above left] {$\bar{\mathbf{u}}_m^{i+1}$} (top);
\draw[very thick,dashed] (left) -- node[above] {$\underline{\bar{\mathbf{s}}}^i_m$} (right);
\draw (bottom) node [draw, fill=gray!12,rounded corners] {$\bar{\beta}(i,m)\cdot \bar{X}^{i-1}$};
\draw (left) node [draw, fill=gray!12,rounded corners] {$\bar{\alpha}(i,m-1)\cdot \bar{X}^i$};
\draw (right) node [draw, fill=gray!12,rounded corners] {$\bar{\alpha}(i,m)\cdot \bar{X}^i$};
\draw (top) node [draw, fill=gray!12,rounded corners] {$\bar{X}^{i+1}$};
\end{tikzpicture}
}
}
\caption{$\bar{G}_m^k$-orbit representatives of 2-cells in $\bar{X}_m^k$ in layer $i$.}
\label{fig:smallfun}
\end{figure}

In Figure \ref{fig:smallfun}, the following shorthand is used:
$$
\bar{\alpha}(i,j):= \phi_m^k\left(\alpha(i,j)\right)=\sigma_1\cdots\sigma_j
$$
$$
\bar{\beta}(i,j):= \phi_m^k\left(\beta(i,j)\right)=\sigma_1^2\cdots\sigma_{j-1}^2\sigma_j
$$
$$
\bar{\gamma}(i,j):= \phi_m^k\left(\gamma(i,j)\right)=\sigma_1^{i+1}\cdots\sigma_j^{i+1}
$$

Given a hyperplane $[\mathbf{u}]$, we define the hyperplane stabiliser as
$$
\text{Stab}\left([\mathbf{u}]\right):=\left\{g\in \bar{G}_m^k\ \middle|\ g\cdot\mathbf{u}\sim\mathbf{u}\right\}.
$$

We can compute the hyperplane stabilisers by observing which $\bar{G}_m^k$-coefficients of an edge type in a particular layer lie in the same hyperplane. This is done by first observing that in order to ``move'' to a different layer and remain in the same hyperplane (for edge $\bar{\mathbf{u}}_j^i$, say), we must utilise squares $\underline{\bar{\mathbf{s}}}_j$ and $\underline{\bar{\mathbf{s}}}_{j-1}$ (using cyclic indexing), as they are the only ones to contain an edge of this type. We can think of the hyperplane stabilisers as fundamental groups of loop spaces where the vertices represent heights and edge labels on the loop space represent how the coefficient changes when ``jumping'' across a square, as shown in Figure \ref{fig:loops}. These edge labels are calculated from the coefficients on the edges in Figure \ref{fig:smallfun}. The direction of the arrow shows the direction in which the edge label is multiplied (take the inverse for the opposite direction). For example, by looking at the top left and bottom right edges of square $g\cdot\underline{\bar{\mathbf{s}}}_1^i$, we can deduce that using this square to drop down to layer $i$, the coefficient of edge $g\cdot\bar{\mathbf{u}}_1^{i+1}$ will be multiplied by $\bar{\alpha}(i,1)=\sigma_1$, resulting in the edge $g\sigma_1\cdot \bar{\mathbf{u}}_1^i$, which is still in the same hyperplane as $g\cdot\bar{\mathbf{u}}_1^{i+1}$. Note that this multiplication should be thought of as occurring on the right, but $\bar{G}_m^k$ is abelian, so we do not need to worry about this.

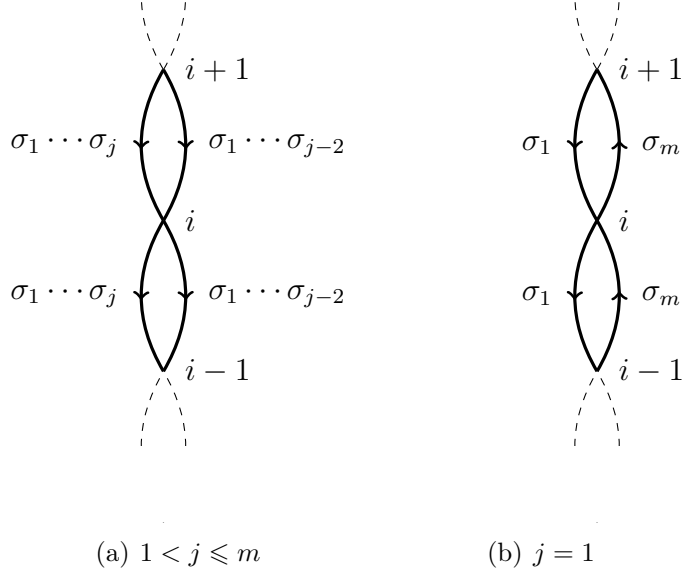
\begin{figure}[h!]
\makebox[\linewidth][c]{
\centering
\subfigure[$1<j\leqslant m$]{
\begin{tikzpicture}
\coordinate (U1) at (0,0cm);
\coordinate (U2) at (0,2cm);
\coordinate (U3) at (0,4cm);
\coordinate (U4) at (0,6cm);
\coordinate (U5) at (0,8cm);
\coordinate (Us1) at (1cm,1cm);
\coordinate (Us2) at (-1cm,1cm);
\coordinate (Us3) at (-1cm,0cm);
\coordinate (Us4) at (1cm,0cm);
\coordinate (Ut1) at (-1cm,8cm);
\coordinate (Ut2) at (1cm,8cm);
\coordinate (Ut3) at (1cm,7cm);
\coordinate (Ut4) at (-1cm,7cm);
\draw[dashed] (U1) to[bend right] (U2);
\draw[dashed] (U1) to[bend left] (U2);
\draw[very thick,middlearrow={<}] (U2) to[bend right] node[right] {$\ \sigma_1\cdots\sigma_{j-2}$} (U3);
\draw[very thick,middlearrow={<}] (U2) to[bend left] node[left] {$\sigma_1\cdots\sigma_j\ $} (U3);
\draw[very thick,middlearrow={<}] (U3) to[bend right] node[right] {$\ \sigma_1\cdots\sigma_{j-2}$} (U4);
\draw[very thick,middlearrow={<}] (U3) to[bend left] node[left] {$\sigma_1\cdots\sigma_j\ $} (U4);
\draw[dashed] (U4) to[bend right] (U5);
\draw[dashed] (U4) to[bend left] (U5);
\fill[white] (Us1) -- (Us2) -- (Us3) -- (Us4) -- cycle;
\fill[white] (Ut1) -- (Ut2) -- (Ut3) -- (Ut4) -- cycle;
\draw (U2) node[right] {$\ i-1$};
\draw (U3) node[right] {$\ i$};
\draw (U4) node[right] {$\ i+1$};
\end{tikzpicture}
}
\subfigure[$j=1$]{
\begin{tikzpicture}
\coordinate (U1) at (0,0cm);
\coordinate (U2) at (0,2cm);
\coordinate (U3) at (0,4cm);
\coordinate (U4) at (0,6cm);
\coordinate (U5) at (0,8cm);
\coordinate (Us1) at (1cm,1cm);
\coordinate (Us2) at (-1cm,1cm);
\coordinate (Us3) at (-1cm,0cm);
\coordinate (Us4) at (1cm,0cm);
\coordinate (Ut1) at (-1cm,8cm);
\coordinate (Ut2) at (1cm,8cm);
\coordinate (Ut3) at (1cm,7cm);
\coordinate (Ut4) at (-1cm,7cm);
\draw[dashed] (U1) to[bend right] (U2);
\draw[dashed] (U1) to[bend left] (U2);
\draw[very thick,middlearrow={>}] (U2) to[bend right] node[right] {$\ \sigma_m$} (U3);
\draw[very thick,middlearrow={<}] (U2) to[bend left] node[left] {$\sigma_1\ $} (U3);
\draw[very thick,middlearrow={>}] (U3) to[bend right] node[right] {$\ \sigma_m$} (U4);
\draw[very thick,middlearrow={<}] (U3) to[bend left] node[left] {$\sigma_1\ $} (U4);
\draw[dashed] (U4) to[bend right] (U5);
\draw[dashed] (U4) to[bend left] (U5);
\fill[white] (Us1) -- (Us2) -- (Us3) -- (Us4) -- cycle;
\fill[white] (Ut1) -- (Ut2) -- (Ut3) -- (Ut4) -- cycle;
\draw (U2) node[right] {$\ i-1$};
\draw (U3) node[right] {$\ i$};
\draw (U3) node[left] {$\ \ \ \ \ \ \ \ \ \ \ \ \ \ \ \ \ \ $};
\draw (U4) node[right] {$\ i+1$};
\end{tikzpicture}
}
}
\caption{Moving between edges of different heights in the same hyperplane of type $\bar{\mathbf{u}}_j$.}
\label{fig:loops}
\end{figure}

From this, (using cyclic indexing) we can determine:

$$
\text{Stab}\left([\bar{\mathbf{u}}_{j}^i]\right)=\left<\sigma_{j-1}\sigma_j\right>.
$$

Note that because squares of a fixed type are in free orbit and our group of coefficients is abelian, it does not matter which particular hyperplane we are in: the stabiliser will only depend on the type of edge. So we can ignore the coefficient or height of an edge when considering its hyperplane stabiliser, which will be mirrored in the notation from now on.

\begin{lemma}
\label{climb}
For integers $a,b$ and elements $g,h\in\bar{G}_m^k$, we have:
$$
g\cdot\bar{\mathbf{u}}_j^a\sim h\cdot\bar{\mathbf{u}}_j^b \implies h \in g\left(\sigma_1\cdots\sigma_j\right)^{a-b} \text{Stab}\left([\bar{\mathbf{u}}_{j}]\right).
$$
\end{lemma}

\begin{proof}
Consider which $\bar{\mathbf{u}}_j$-type edges lie on layer $b$ and are part of the hyperplane $[g\cdot\bar{\mathbf{u}}_j^a]$. Their $\bar{G}_m^k$-coefficient will be $g$ multiplied by $\left(\sigma_1\cdots\sigma_j\right)^{a-b}$ from moving to layer $b$, by moving along the left side of the loop space in Figure \ref{fig:loops}, and then maybe also moved around in the layer by multiplication by an element from the stabiliser, giving the result. It did not matter that we chose to take the left side, as the stabiliser is the fundamental group of the loop space.
\end{proof}

\section{Linear Characters}
\label{sec:6}

Fix a primitive $k$th root of unity $\mu$ in $\mathbb{C}$. For each $1\leqslant j\leqslant m$, consider the homomorphism:

\begin{align*}
\mathfrak{D}(j): \bar{G}_m^k&\to \mathbb{C}\\
\sigma_i&\mapsto \begin{cases*}
\mu & \text{if }$\ i=j$\\
1 & \text{else}
\end{cases*}
\end{align*}

This defines a linear character. We will show that $\bar{X}_m^k$ is a special cube complex by working with $\text{Ch}(\bar{G}_m^k)$, the abelian group of linear characters (under pointwise multiplication). In particular, we can multiply and take inverses. The main idea is to come up with characters constant on certain sets. Cyclic indexing will again be used in this section, so for example, $\mathfrak{D}(j-1)$ for $j=1$ will refer to $\mathfrak{D}(m)$.

We will proceed by checking each of the 4 conditions for a special cube complex:

\medskip

\textbf{1.} ($\sim$,$\perp$) Every square in $\bar{X}_m^k$ is such that every corner consists of two edges of different types meeting, hence a hyperplane can never cross itself.

\medskip

\textbf{2.} (2-sided) We choose an orientation for each edge by deciding to make each edge point upwards. Now each square takes positively oriented edges to positively oriented edges or negatively oriented edges to negatively oriented edges, so we can never have the situation $\bar{\mathbf{u}}_j^i\sim-\bar{\mathbf{u}}_j^i$.

\medskip

\textbf{3.} ($\sim$,$\circlearrowright$) Next we investigate whether a hyperplane can self-osculate. In order for this to occur, we must have two distinct edges $g\cdot\bar{\mathbf{u}}_j^a$ and $h\cdot\bar{\mathbf{u}}_j^b$ for some integers $a,b$, some $1\leqslant j\leqslant m$ and $g,h\in\bar{G}_m^k$, for which we have
$$
g\cdot\bar{\mathbf{u}}_j^a\sim h\cdot\bar{\mathbf{u}}_j^b\ \ \ \text{and also}\ \ g\cdot \bar{\mathbf{u}}_j^a\circlearrowright h\cdot \bar{\mathbf{u}}_j^b.
$$

We already know from Lemma \ref{climb} that the first condition implies that

$$
h \in g\left(\sigma_1\cdots\sigma_j\right)^{a-b} \text{Stab}\left([\bar{\mathbf{u}}_{j}]\right).
$$

Consider the second condition. Since the two edges in question osculate, they must share a common vertex. This implies that $b\in\{a-1,a,a+1\}$. We consider each in turn:

\medskip

\textbf{Case 1:} $(b=a-1)$ From the fundamental domain in the previous section, we know that edge $\sigma_1^{-1}\cdots\sigma_{j-1}^{-1}\cdot\bar{\mathbf{u}}_j^{i+1}$ is attached to vertex $\bar{X}^{i}$ and so is edge $\bar{\mathbf{u}}_j^i$. By translation, this means that to get at all the edges of type $\bar{\mathbf{u}}_j^{a-1}$ which osculate with $g\cdot\bar{\mathbf{u}}_j^a$, the coefficient is multiplied by $\sigma_1\cdots\sigma_{j-1}$ and also possibly by some stabiliser of the vertex (note that since $\bar{G}_m^k$ is abelian, the stabiliser of a vertex is determined only by its height). This implies that

$$
h=\left(\sigma_1^{a-1}\cdots\sigma_m^{a-1}\right)^c\sigma_1\cdots\sigma_{j-1}g
$$

for some integer $c$. So in order for the self-osculation to be possible, it must be true that

$$
\left(\sigma_1^{a-1}\cdots\sigma_m^{a-1}\right)^c\sigma_1\cdots\sigma_{j-1} \in \left(\sigma_1\cdots\sigma_j\right)^{a-b} \text{Stab}\left([\bar{\mathbf{u}}_{j}]\right).
$$

However, the character $\mathfrak{D}(j-1)\mathfrak{D}(j)^{-1}$ takes the value 1 on the set on the right and takes the value $\mu$ on the element on the left, hence this is not possible.

\medskip

\textbf{Case 2:} $(b=a+1)$ Reasoning similarly to above, we have that, for some integer $c$,

$$
h=\left(\sigma_1^a\cdots\sigma_m^a\right)^c\sigma_1^{-1}\cdots\sigma_{j-1}^{-1}g.
$$

This implies that for self-osculation, we need

$$
\left(\sigma_1^a\cdots\sigma_m^a\right)^c\sigma_1^{-1}\cdots\sigma_{j-1}^{-1} \in \left(\sigma_1\cdots\sigma_j\right)^{a-b} \text{Stab}\left([\bar{\mathbf{u}}_{j}]\right).
$$

However, the character $\mathfrak{D}(j-1)\mathfrak{D}(j)^{-1}$ takes the value 1 on the set on the right and takes the value $\mu^{-1}$ on the element on the left, hence this is not possible.

\medskip

\textbf{Case 3:} $(b=a)$ Here we can have two further possibilities: the two edges could be joined at a vertex of height $a$ or $a-1$. Note that in either case, this requires this vertex to be a branching vertex, hence the respective heights are not divisible by $k$.

In the former case, we get $h$ from $g$ by multiplying by a stabiliser of a vertex of height $a$, so for some integer $c\not\in k\mathbb{Z}$ (because we want the two edges to be distinct) we have:

$$
h=\left(\sigma_1^a\cdots\sigma_m^a\right)^cg.
$$

This implies that for self-osculation, we need

$$
\left(\sigma_1^a\cdots\sigma_m^a\right)^c \in \left(\sigma_1\cdots\sigma_j\right)^{a-b} \text{Stab}\left([\bar{\mathbf{u}}_{j}]\right).
$$

However, since $a=b$, this is the same as asking for $\left(\sigma_1^a\cdots\sigma_m^a\right)^c$ to be in $\left<\sigma_{j-1}\sigma_j\right>$. Since $m>2$, the character $\mathfrak{D}(j+1)$ evaluates to 1 on the stabiliser, however it takes the value $\mu^{ac}$ on $\left(\sigma_1^a\cdots\sigma_m^a\right)^c$. These do not agree, as $k$ is prime and neither of $a$ or $c$ are divisible by $k$.

In the latter case, we get $h$ from $g$ by multiplying by a stabiliser of a vertex of height $a-1$. Similarly, we obtain, for some integer $c\not\in k\mathbb{Z}$,

$$
h=\left(\sigma_1^{a-1}\cdots\sigma_m^{a-1}\right)^cg.
$$

Just as in the previous case, this is the same as asking for $\left(\sigma_1^{a-1}\cdots\sigma_m^{a-1}\right)^c$ to be in $\left<\sigma_{j-1}\sigma_j\right>$. The character $\mathfrak{D}(j+1)$ evaluates to 1 on the stabiliser, but takes the value $\mu^{c(a-1)}$ on the element. These again do not coincide as $k$ is prime and neither of $c$ or $a-1$ are divisible by $k$ in this case.

Therefore we have showed that no hyperplane self-osculates in $\bar{X}_m^k$.

\medskip

\textbf{4.} ($\perp$,$\circlearrowright$) Finally, we investigate whether two hyperplanes $\mathfrak{H}$, $\mathfrak{H}'$ can inter-osculate. In order for this to occur, we must have two distinct hyperplanes which cross. This can only occur between one of type $[\bar{\mathbf{u}}_j]$ and one of type $[\bar{\mathbf{u}}_{j+1}]$ (using cyclic indexing). In fact, since an intersection of hyperplanes can only happen at a square, it does not matter which corner of the square or which edges of that square are considered, since each edge in a pair of parallel edges represents the same hyperplane. Hence without loss of generality we may assume that it is the top corner. Now we proceed according to the two ways this corner can be, according to the fundamental domain in Figure \ref{fig:smallfun} from Section \ref{sec:5}.

\medskip

\textbf{Case 1:} $(1\leqslant j<m)$ In this situation, $\mathfrak{H}\perp\mathfrak{H}'$ comes from edges $g\cdot\bar{\mathbf{u}}_j^a$ and $g\cdot\bar{\mathbf{u}}_{j+1}^a$ for some integer $a$ and $g\in\bar{G}_m^k$. As when dealing with the self-osculation, we now have 4 possible sources of osculation to deal with:

\begin{itemize}
\item \textbf{Sub-case 1.1:} (joined at the top vertex) From the top corner of the fundamental domain we can deduce that osculation comes from edges $h\cdot\bar{\mathbf{u}}_j^b$ and $h\left(\sigma_1^b\cdots\sigma_m^b\right)^c\cdot\bar{\mathbf{u}}_{j+1}^b$, for some integers $b,c\not\in k\mathbb{Z}$ (because we are attached at a branching vertex and the edges are not meeting at the corner of a square, respectively) and some $h\in\bar{G}_m^k$. We now need to check if $h\cdot\bar{\mathbf{u}}_j^b\sim g\cdot\bar{\mathbf{u}}_j^a$ and $h\left(\sigma_1^b\cdots\sigma_m^b\right)^c\cdot\bar{\mathbf{u}}_{j+1}^b\sim g\cdot\bar{\mathbf{u}}_{j+1}^a$ are both possible simultaneously. By Lemma \ref{climb}, this implies:
$$
g\in \left(h\left(\sigma_1\cdots\sigma_j\right)^{b-a}\text{Stab}\left([\bar{\mathbf{u}}_j]\right)\right)\cap \left( h\left(\sigma_1^b\cdots\sigma_m^b\right)^c\left(\sigma_1\cdots\sigma_{j+1}\right)^{b-a}\text{Stab}\left([\bar{\mathbf{u}}_{j+1}]\right) \right),
$$
so we need
$$
\left(\left(\sigma_1\cdots\sigma_j\right)^{b-a}\left<\sigma_{j-1}\sigma_j\right>\right)\cap\left( \left(\sigma_1^b\cdots\sigma_m^b\right)^c\left(\sigma_1\cdots\sigma_{j+1}\right)^{b-a}\left<\sigma_j\sigma_{j+1}\right> \right)
$$
to not be empty, which is the same as
$$
\left<\sigma_{j-1}\sigma_j\right>\cap \left(\sigma_1^b\cdots\sigma_m^b\right)^c\sigma_{j+1}^{b-a}\left<\sigma_j\sigma_{j+1}\right> 
$$
not being empty. However, as $m>3$, the character $\mathfrak{D}(j+2)$ evaluates to 1 on the left set and to $\mu^{bc}\neq1$ on the right set, so the intersection is empty.
\item \textbf{Sub-case 1.2:} (joined at the bottom vertex) From the bottom corner of the fundamental domain we can deduce that osculation comes from edges $h\sigma_j\cdot\bar{\mathbf{u}}_j^{b+1}$ and $h\left(\sigma_1^b\cdots\sigma_m^b\right)^c\cdot\bar{\mathbf{u}}_{j+1}^{b+1}$ for some integers $b,c\not\in k\mathbb{Z}$ and some $h\in\bar{G}_m^k$. Similarly to above, this implies that
$$
\sigma_j\left<\sigma_{j-1}\sigma_j\right>\cap\left(\sigma_1^b\cdots\sigma_m^b\right)^c\sigma_{j+1}^{b-a+1}\left<\sigma_j\sigma_{j+1}\right>
$$
is not empty. However, as $m>3$, the character $\mathfrak{D}(j+2)$ evaluates to 1 on the left set and to $\mu^{bc}\neq1$ on the right set, so the intersection is empty.
\item \textbf{Sub-case 1.3:} ($j$ above $j+1$) From the left corner of the fundamental domain we can deduce that osculation comes from edges $h\left(\sigma_1^b\cdots\sigma_m^b\right)^c\cdot\bar{\mathbf{u}}_j^{b+1}$ and $h\sigma_1\cdots\sigma_{j-1}\cdot\bar{\mathbf{u}}_{j+1}^b$ for some integers $b,c\not\in k\mathbb{Z}$ and some $h\in\bar{G}_m^k$. Similarly to above, this implies that
$$
\left(\sigma_1^b\cdots\sigma_m^b\right)^c\sigma_j\left<\sigma_{j-1}\sigma_j\right>\cap\sigma_{j+1}^{b-a}\left<\sigma_j\sigma_{j+1}\right>
$$
is not empty. However, as $m>3$, the character $\mathfrak{D}(j+2)$ evaluates to 1 on the right set and to $\mu^{bc}\neq1$ on the left set, so the intersection is empty.
\item \textbf{Sub-case 1.4:} ($j$ below $j+1$) From the right corner of the fundamental domain we can deduce that osculation comes from edges $h\left(\sigma_1^b\cdots\sigma_m^b\right)^c\cdot\bar{\mathbf{u}}_{j+1}^{b+1}$ and $h\sigma_1\cdots\sigma_j\cdot\bar{\mathbf{u}}_j^b$ for some integers $b,c\not\in k\mathbb{Z}$ and some $h\in\bar{G}_m^k$. Similarly to above, this implies that
$$
\left(\sigma_1^b\cdots\sigma_m^b\right)^c\sigma_{j+1}^{b-a+1}\left<\sigma_j\sigma_{j+1}\right>\cap\left<\sigma_{j-1}\sigma_j\right>
$$
is not empty. However, as $m>3$, the character $\mathfrak{D}(j+2)$ evaluates to 1 on the right set and to $\mu^{bc}\neq1$ on the left set, so the intersection is empty.
\end{itemize}

\textbf{Case 2:} $(j=m,``j+1"=1)$ In this situation, $\mathfrak{H}\perp\mathfrak{H}'$ comes from edges $g\cdot\bar{\mathbf{u}}_m^a$ and $g\sigma_1^a\cdots\sigma_m^a\cdot\bar{\mathbf{u}}_{1}^a$ for some integer $a$ and $g\in\bar{G}_m^k$. Once again we have 4 possible sources of osculation to deal with:

\begin{itemize}
\item \textbf{Sub-case 2.1:} (joined at the top vertex) From the top corner of the fundamental domain we can deduce that osculation comes from edges $h\left(\sigma_1^b\cdots\sigma_m^b\right)^c\cdot\bar{\mathbf{u}}_m^b$ and $h\sigma_1^b\cdots\sigma_m^b\cdot\bar{\mathbf{u}}_1^b$ for some integers $b,c\not\in k\mathbb{Z}$ and some $h\in\bar{G}_m^k$. From Lemma 5.1, in order for this to occur, we need both of
$$
g\in h\left(\sigma_1^b\cdots\sigma_m^b\right)^c\left(\bar{x_1}\cdots\sigma_m\right)^{b-a}\left<\sigma_{m-1}\sigma_m\right>
$$
and
$$
g\sigma_1^a\cdots\sigma_m^a\in h\sigma_1^b\cdots\sigma_m^b\sigma_1^{b-a}\left<\sigma_1\sigma_m\right>
$$
to hold simultaneously. This is the same as
$$
\sigma_1^{b-a}\left<\sigma_1\sigma_m\right>\cap\left(\sigma_1^b\cdots\sigma_m^b\right)^c\left<\sigma_{m-1}\sigma_m\right>
$$
not being empty. However, since $m>3$, the character $\mathfrak{D}(2)$ evaluates to 1 on the left set and to $\mu^{bc}\neq1$ on the right set, so the intersection is empty.
\item \textbf{Sub-case 2.2:} (joined at the bottom vertex) From the bottom corner of the fundamental domain we can deduce that osculation comes from edges $h\left(\sigma_1^b\cdots\sigma_m^b\right)^c\cdot\bar{\mathbf{u}}_m^{b+1}$ and $h\sigma_1^{b+1}\cdots\sigma_{m-1}^{b+1}\sigma_m^b\cdot\bar{\mathbf{u}}_1^{b+1}$ for some integers $b,c\not\in k\mathbb{Z}$ and some $h\in\bar{G}_m^k$. Similarly to above, this implies that
$$
\left(\sigma_1^b\cdots\sigma_m^b\right)^c\sigma_m\left<\sigma_{m-1}\sigma_m\right>\cap\sigma_1^{b-a+1}\left<\sigma_1\sigma_m\right>
$$
is not empty. However, since $m>3$, the character $\mathfrak{D}(2)$ evaluates to 1 on the right set and to $\mu^{bc}\neq1$ on the left set, so the intersection is empty.
\item \textbf{Sub-case 2.3:} (1 above $m$) From the right corner of the fundamental domain we can deduce that osculation comes from edges $h\sigma_1^{b+1}\cdots\sigma_m^{b+1}\cdot\bar{\mathbf{u}}_1^{b+1}$ and $h\left(\sigma_1^b\cdots\sigma_m^b\right)^c\sigma_1\cdots\sigma_m\cdot\bar{\mathbf{u}}_m^b$ for some integers $b,c\not\in k\mathbb{Z}$ and some $h\in\bar{G}_m^k$. Similarly to above, this implies that
$$
\sigma_1^{b-a+1}\left<\sigma_1\sigma_m\right>\cap\left(\sigma_1^b\cdots\sigma_m^b\right)^c\left<\sigma_{m-1}\sigma_{m}\right>
$$
is not empty. However, since $m>3$, the character $\mathfrak{D}(2)$ evaluates to 1 on the left set and to $\mu^{bc}\neq1$ on the right set, so the intersection is empty.
\item \textbf{Sub-case 2.4:} (1 below $m$) From the left corner of the fundamental domain we can deduce that osculation comes from edges $h\left(\sigma_1^b\cdots\sigma_m^b\right)^c\cdot\bar{\mathbf{u}}_m^{b+1}$ and $h\sigma_1^{b+1}\cdots\sigma_{m-1}^{b+1}\sigma_m^b\cdot\bar{\mathbf{u}}_1^b$ for some integers $b,c\not\in k\mathbb{Z}$ and some $h\in\bar{G}_m^k$. Similarly to above, this implies that
$$
\left(\sigma_1^b\cdots\sigma_m^b\right)^c\sigma_m\left<\sigma_{m-1}\sigma_m\right>\cap\sigma_1^{b-a}\left<\sigma_1\sigma_m\right>
$$
is not empty. However, since $m>3$, the character $\mathfrak{D}(2)$ evaluates to 1 on the right set and to $\mu^{bc}\neq1$ on the left set, so the intersection is empty. Hence there is no inter-osculation between hyperplanes in $\bar{X}_m^k$.
\end{itemize}

Hence $\bar{X}_m^k$ is a special cube complex.

\medskip

\emph{Proof of Theorem \ref{thm:main}:} We have that $\ker{\phi_m^k}$ is torsion-free and $X_m^k$ is simply-connected, therefore $\ker{\phi_m^k}=\pi_1\left(\bar{X}_m^k\right)$. We also have that $\ker{\phi_m^k}$ is a finite-index subgroup of $G_m^k$, hence $G_m^k$ is virtually special. By Theorem \ref{pre:wise}, $\ker{\phi_m^k}$ embeds into a finitely generated RAAG, so is residually finite. Finally, as $\ker{\phi_m^k}$ is a finite-index subgroup, $G_m^k$ is also residually finite.\hfill \qed

\section{Remarks}
\label{sec:7}

\textbf{1.} Note that the abelianisation of $G_m^k$ is
$$
{G_m^k}^{ab}=\mathbb{Z}^m / \left<(k,k,\dots,k)\right>.
$$
Using Smith Normal Form, we get
$$
\underbrace{(k,k,\dots,k)}_\text{$m$}\mapsto(k,\underbrace{0,\dots,0}_\text{$m-1$})\implies {G_m^k}^{ab}\cong C_k\times\mathbb{Z}^{m-1}
$$
and hence $G_m^k$ are distinct up to isomorphism for different integer pairs $(m,k)$.

\medskip

\textbf{2.} The action of $G_m^k$ on the CAT(0) cube complex $X_m^k$ is interesting, because even though there are only finitely many $G_m^k$-orbits of hyperplanes, we get infinitely many $\text{Stab}\left(\mathfrak{H}\right)$-orbits of hyperplanes crossing a given hyperplane $\mathfrak{H}$. This makes it difficult to apply existing tools such as Theorem 4.1 in \cite{coxeter} for proving that the group is virtually special.

For an explicit example, consider the hyperplane $[\mathbf{u}_2^0]$. From the fundamental domain in section 4, we have $\mathbf{u}_2^0\sim\mathbf{u}_2^i\ \forall\ i\in\mathbb{Z}$ and also $\mathbf{u}_2^i\perp\mathbf{u}_3^i\ \forall\ i\in\mathbb{Z}$. Similarly to Section \ref{sec:5} we can compute the hyperplane stabilisers
$$
\text{Stab}\left([\mathbf{u}_2]\right)=\left<x_1^a x_2^a\ |\ \forall\ a\in\mathbb{Z}\right>,\ \text{Stab}\left([\mathbf{u}_3]\right)=\left<x_2^a x_3^a\ |\ \forall\ a\in\mathbb{Z}\right>,
$$
as well as $\mathbf{u}_3^i\sim x_1^i\cdot\mathbf{u}_3^0$. This implies that the set of coefficients $g$ for which $[\mathbf{u}_2^0]$ crosses $[g\cdot\mathbf{u}_3^0]$ is $I=\{x_1^i\ |\ i\in\mathbb{Z}\}$. Two elements of $I$ give the same hyperplane in the $\text{Stab}\left([\mathbf{u}_2^0]\right)$-orbit of hyperplanes which cross $[\mathbf{u}_2^0]$ if one can get from one to the other by multiplying on the left by an element from $\text{Stab}\left([\mathbf{u}_2]\right)$ and on the right by an element from $\text{Stab}\left([\mathbf{u}_3]\right)$. If, after this identification, there would be only a finite list of elements left, then there would also be a finite list of such elements in the abelianisation. However, in ${G_m^k}^{ab}$ this consists of all elements of the form $x_1^{i+a}x_2^{a+b}x_3^b$ for some integers $a,b,i$. To remain within $I$, we require $b=a=0$, so actually there are infinitely many elements in the abelianisation after identification.

This means that there are infinitely many $\text{Stab}\left([\mathbf{u}_2^0]\right)$-orbits of hyperplanes crossing $[\mathbf{u}_2^0]$. Similar infinite behaviour can occur with other hyperplanes and with osculation.

\medskip

\textbf{3.} The groups in Theorem \ref{thm:main} are not finitely presented because of Corollary 14.5 in \cite{uncountably}. Lemma 14.4 applies with the modification that a relation $x_1^i\cdots x_m^i$ becomes $\left(x_1^i\cdots x_m^i\right)^k$ when $i\not\in S$.

\medskip

\textbf{4.} The special groups which are the finite index subgroups can be naturally thought of as kernels of maps to $\mathbb{Z}$ in the following way.

In a more general setting of a virtually torsion-free generalised Bestvina-Brady group $G$ of finite ramification with associated locally finite $\text{CAT}(0)$ cube complex $X$, we have the following short exact sequence:
$$
1\to \ker{\phi}\to G\to \bar{G}\to 1,
$$
where $\bar{G}$ is some finite group and $\ker{\phi}$ is torsion-free. From Theorem \ref{crit}, the associated branching set is periodic, hence the complex $X$ is periodic. Therefore $X/\ker{\phi}$ is a $\mathbb{Z}$-cover of some finite cube complex $\chi$ and we also get a short exact sequence
$$
1\to \ker{\phi}\to \pi_1(\chi)\to \mathbb{Z}\to 1.
$$

The finite complex $\chi$ can be thought of as a branched Salvetti complex, because in the case of a Bestvina-Brady group, this would actually be a Salvetti complex and $\ker{\phi}$ the associated Bestvina-Brady group. It is the author's intention to study properties of such complexes further, as well as classify groups in \hyperlink{theset}{$\mathfrak{R}$}, the set of generalised Bestvina-Brady groups of finite ramification which are virtually special.

\section*{Appendix}
\hypertarget{apx}{}

Recall the following facts from Section \ref{sec:4}:

\medskip

(1.) A part of the descending link of $X^i$ is shown in Figure \ref{fig:link2}.

\begin{figure}[h!]
\centering
\begin{tikzpicture}
\coordinate (left) at (0,0);
\coordinate (mid1) at (2.5cm,0);
\coordinate (mid2) at (5cm,0);
\coordinate (right) at (7.6cm,0);
\coordinate (rightest) at (10.7cm,0);
\draw[very thick] (left) -- node[below] {$\underline{\mathbf{s}}_1^i$} (mid1);
\draw[very thick] (rightest) -- node[below left] {$\underline{\mathbf{s}}_m^i$} (right);
\draw[very thick] (mid2) -- node[below] {$\underline{\mathbf{s}}_{m-1}^i$} (right);
\draw[very thick,dashed] (mid1) -- (mid2);
\draw (left) node [draw, fill=gray!12,rounded corners] {$\mathbf{u}_1^i$};
\draw (mid1) node [draw,fill=gray!12,rounded corners] {$\mathbf{u}_2^i$};
\draw (mid2) node [draw,fill=gray!12,rounded corners] {$\mathbf{u}_{m-1}^i$};
\draw (right) node [draw,fill=gray!12,rounded corners] {$\mathbf{u}_{m}^i$};
\draw (rightest) node [draw,fill=gray!12,rounded corners] {$x_1^i\cdots x_m^i\cdot\mathbf{u}_{1}^i$};
\end{tikzpicture}
\caption{Part of the descending link of $X^i$.}
\label{fig:link2}
\end{figure}
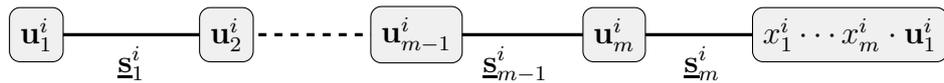

(2.) Vertex $X^i$, for positive $i$, has the following vertices on the boundary of its shadow:

\begin{align*}
 & & & & & & X^0 &,\\
 x_1\cdot X^0&,& x_1^2\cdot X^0&,& \dots&\ ,& x_1^i\cdot X^0&,\\
  x_1^i x_2\cdot X^0&,& x_1^i x_2^2\cdot X^0&,& \dots&\ ,& x_1^i x_2^i\cdot X^0&,\\
  & & \vdots& & & & & \\
  x_1^i x_2^i \cdots x_{m-1}^i x_m\cdot X^0&,& x_1^i x_2^i \cdots x_{m-1}^i x_m^2\cdot X^0&,& \dots&\ ,& x_1^i x_2^i\cdots x_m^i\cdot X^0.&
\end{align*}

(3.) Vertex $X^i$, for negative $i$, has the following vertices on the boundary of its shadow:

\begin{align*}
 & & & & & & X^0 &,\\
 x_1^{-1}\cdot X^0&,& x_1^{-2}\cdot X^0&,& \dots&\ ,& x_1^i\cdot X^0&,\\
  x_1^i x_2^{-1}\cdot X^0&,& x_1^i x_2^{-2}\cdot X^0&,& \dots&\ ,& x_1^i x_2^i\cdot X^0&,\\
  & & \vdots& & & & & \\
  x_1^i x_2^i \cdots x_{m-1}^i x_m^{-1}\cdot X^0&,& x_1^i x_2^i \cdots x_{m-1}^i x_m^{-2}\cdot X^0&,& \dots&\ ,& x_1^i x_2^i\cdots x_m^i\cdot X^0&.
\end{align*}

We prove Lemma \ref{funlemma} by induction on $j$.

\medskip

\textbf{Base case:} ($j=1$) Applying (2.) to vertex $X^1$ in combination with (1.) yields:

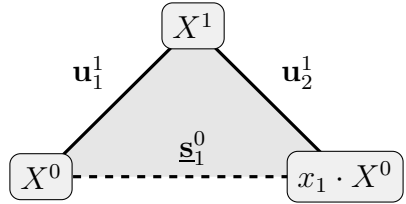
\begin{figure}[h]
\centering
\begin{tikzpicture}
\coordinate (top) at (0,2cm);
\coordinate (left) at (-2cm,0);
\coordinate (right) at (2cm,0);
\fill[gray!20] (left) -- (top) -- (right) -- cycle;
\draw[very thick] (right) -- node[above right] {$\mathbf{u}_{2}^{1}$} (top);
\draw[very thick] (left) -- node[above left] {$\mathbf{u}_{1}^{1}$} (top);
\draw[very thick,dashed] (left) -- node[above] {$\underline{\mathbf{s}}^0_1$} (right);
\draw (left) node [draw, fill=gray!12,rounded corners] {$X^0$};
\draw (right) node [draw, fill=gray!12,rounded corners] {$x_1\cdot X^0$};
\draw (top) node [draw, fill=gray!12,rounded corners] {$X^1$};
\end{tikzpicture}
\end{figure}

Applying (3.) to vertex $X^{-1}$ and keeping in mind that edge $\mathbf{u}_1^0$ joins $X^{-1}$ to $X^0$ yields:

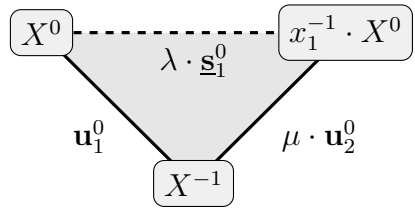
\begin{figure}[h]
\centering
\begin{tikzpicture}
\coordinate (bottom) at (0,-2cm);
\coordinate (left) at (-2cm,0);
\coordinate (right) at (2cm,0);
\fill[gray!20] (left) -- (bottom) -- (right) -- cycle;
\draw[very thick] (right) -- node[below right] {$\mu\cdot\mathbf{u}_{2}^{0}$} (bottom);
\draw[very thick] (left) -- node[below left] {$\mathbf{u}_{1}^{0}$} (bottom);
\draw[very thick,dashed] (left) -- node[below] {$\lambda\cdot\underline{\mathbf{s}}^0_1$} (right);
\draw (left) node [draw, fill=gray!12,rounded corners] {$X^0$};
\draw (right) node [draw, fill=gray!12,rounded corners] {$x_1^{-1}\cdot X^0$};
\draw (bottom) node [draw, fill=gray!12,rounded corners] {$X^{-1}$};
\end{tikzpicture}
\end{figure}

for some $\mu,\lambda\in G_m^k$. Since 0 is not a branching layer, $X^0$ has a unique edge from $\text{Orb}\left(\mathbf{u}_2^0\right)$ in its descending link. From (1.) we know that this is $\mathbf{u}_2^0$, which means that $\mu = x_1^{-1}$. This also means that $\lambda=x_1^{-1}$, since square $\underline{\mathbf{s}}_1^0$ has edge $\mathbf{u}_2^0$ attached underneath $X^0$. Hence square $\underline{\mathbf{s}}_1^0$ is as claimed in Lemma \ref{funlemma}. We now proceed to proving the general form of square $\underline{\mathbf{s}}_1^i$ by induction, depending on whether $i$ is positive or negative.

\medskip

\textbf{Case 1:} ($i$ positive) Claim: Square $\underline{\mathbf{s}}_1^i$ is as in Lemma \ref{funlemma} for $i>0$.
\begin{proof}
Base case: ($i=1$) Considering the pyramid with apex $X^2$ and applying (2.) along with (1.) gives Figure \ref{j1i2}.

\begin{figure}[h]
\centering
\begin{tikzpicture}
\coordinate (c1) at (2cm,2cm);
\coordinate (c2) at (6cm,2cm);
\coordinate (c3) at (10cm,2cm);
\coordinate (d1) at (4cm,4cm);
\coordinate (d2) at (8cm,4cm);
\coordinate (apex) at (6cm,6cm);

\fill[gray!20] (apex) -- (c3) -- (c1) -- cycle;

\draw[very thick] (apex) -- node[above left] {$\mathbf{u}_1^2$} (d1);
\draw[very thick] (apex) -- node[above right] {$\mathbf{u}_2^2$} (d2);
\draw[very thick] (d1) -- node[above left] {$\mathbf{u}_1^1$} (c1);
\draw[very thick] (d2) -- node[below right] {$\delta\cdot \mathbf{u}_1^1$} (c2);
\draw[dashed] (d2) -- (c3);
\draw[very thick] (d1) -- node[below left] {$\mathbf{u}_2^1$} (c2);

\draw[very thick,dashed] (d1) -- node[above] {$\underline{\mathbf{s}}_1^1$} (d2);
\draw[very thick,dashed] (c1) -- node[below] {$\underline{\mathbf{s}}_1^0$} (c2);
\draw[dashed] (c2) -- (c3);

\draw (apex) node [draw, fill=gray!10,rounded corners] {$X^2$};
\draw (d1) node [draw, fill=gray!10,rounded corners] {$X^1$};
\draw (d2) node [draw, fill=gray!10,rounded corners] {$\gamma\cdot X^1$};
\draw (c1) node [draw, fill=gray!10,rounded corners] {$X^0$};
\draw (c2) node [draw, fill=gray!10,rounded corners] {$x_1\cdot X^0$};
\draw (c3) node [draw, fill=gray!10,rounded corners] {$x_1^2\cdot X^0$};
\end{tikzpicture}
\caption{Figuring out the square $\underline{\mathbf{s}}_1^1$.}
\label{j1i2}
\end{figure}

Now by considering which translate of the square $\underline{\mathbf{s}}_1^0$ fits in the bottom-right, we determine that $\delta=\gamma=x_1$. This completes the square $\underline{\mathbf{s}}_1^1$.

Inductive step: ($i>1$) Assume that the claim has already been proven for all $i$ up to and including $n\geqslant 1$. Consider the pyramid with apex $X^{n+2}$. Applying (1.) gives the top half of square $\underline{\mathbf{s}}_1^{n+1}$. Applying (2.) gives the vertices in layer 0 at the bottom of the face of the pyramid, and working upwards using the inductive hypothesis completes square $\underline{\mathbf{s}}_1^{n+1}$.

\end{proof}

\textbf{Case 2:} ($i$ negative) Claim: Square $\underline{\mathbf{s}}_1^{-i}$ is as in Lemma \ref{funlemma} for $i>0$.
\begin{proof}

Base case: ($i=1$) Consider the pyramid with apex $X^{-2}$ and apply (2.). Note that we know what the line from $X^{-2}$ to $X^0$ (where two faces of the pyramid meet) looks like, because by convention this is a series of edges of type $\mathbf{u}_1$. By also considering which translates of $\underline{\mathbf{s}}_1^0$-squares go at the top, we get Figure \ref{j1im2}.

\begin{figure}[h]
\centering
\begin{tikzpicture}
\coordinate (b1) at (0cm,0cm);
\coordinate (b2) at (4cm,0cm);
\coordinate (b3) at (8cm,0cm);
\coordinate (d1) at (2cm,-2cm);
\coordinate (d2) at (6cm,-2cm);
\coordinate (apex) at (4cm,-4cm);

\fill[gray!20] (apex) -- (b3) -- (b1) -- cycle;

\draw[very thick] (apex) -- node[below left] {$\mathbf{u}_1^{-1}$} (d1);
\draw[very thick] (apex) -- node[below right] {$\mu\cdot\mathbf{u}_2^{-1}$} (d2);
\draw[very thick] (d1) -- node[below left] {$\mathbf{u}_1^0$} (b1);
\draw[very thick] (d1) -- node[above left] {$x_1^{-1}\cdot\mathbf{u}_2^0$} (b2);
\draw[very thick] (b2) -- node[above right] {$x_1^{-1}\cdot\mathbf{u}_1^0$} (d2);
\draw[dashed] (d2) -- (b3);
\draw[very thick, dashed] (b1) -- node[above] {$x_1^{-1}\cdot\underline{\mathbf{s}}_1^0$} (b2);
\draw[very thick, dashed] (b2) -- node[above] {$x_1^{-2}\cdot\underline{\mathbf{s}}_1^0$} (b3);
\draw[very thick, dashed] (d1) -- node[below] {$\alpha\cdot\underline{\mathbf{s}}_1^{-1}$} (d2);
\draw (b1) node [draw, fill=gray!10,rounded corners] {$ X^0$};
\draw (b2) node [draw, fill=gray!10,rounded corners] {$x_1^{-1}\cdot X^0$};
\draw (b3) node [draw, fill=gray!10,rounded corners] {$x_1^{-2}\cdot X^0$};
\draw (d1) node [draw, fill=gray!10,rounded corners] {$X^{-1}$};
\draw (d1) node [draw, fill=gray!10,rounded corners] {$X^{-1}$};
\draw (d2) node [draw, fill=gray!10,rounded corners] {$x_1^{-1}\cdot X^{-1}$};
\draw (apex) node [draw, fill=gray!10,rounded corners] {$X^{-2}$};
\end{tikzpicture}
\caption{Figuring out the square $\underline{\mathbf{s}}_1^{-1}$.}
\label{j1im2}
\end{figure}

Now by applying (1.) to vertex $X^0$, we can determine that $\alpha=x_1^{-1}$. In order to compute $\mu$, we must take into account the full link of vertex $X^{-1}$. Using (1.) applied to vertex $X^0$, the square $x_1^{-1}\cdot\underline{\mathbf{s}}_1^0$ and Figure \ref{j1im2}, we can deduce a part of the full link of vertex $X^{-1}$, shown in Figure \ref{j1xm1}. Note that labels on the edges in the link signify which squares these edges of the link are originating from.

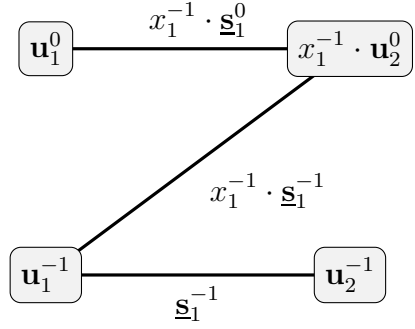
\begin{figure}
\centering
\begin{tikzpicture}
\coordinate (bl) at (0cm,0cm);
\coordinate (br) at (4cm,0cm);
\coordinate (tl) at (0cm,3cm);
\coordinate (tr) at (4cm,3cm);
\draw[very thick] (tl) -- node[above] {$x_1^{-1}\cdot\underline{\mathbf{s}}_1^0$} (tr);
\draw[very thick] (tr) -- node[below right] {$x_1^{-1}\cdot\underline{\mathbf{s}}_1^{-1}$} (bl);
\draw[very thick] (bl) -- node[below] {$\underline{\mathbf{s}}_1^{-1}$} (br);
\draw (bl) node [draw, fill=gray!10,rounded corners] {$\mathbf{u}_1^{-1}$};
\draw (br) node [draw, fill=gray!10,rounded corners] {$\mathbf{u}_2^{-1}$};
\draw (tr) node [draw, fill=gray!10,rounded corners] {$x_1^{-1}\cdot\mathbf{u}_2^{0}$};
\draw (tl) node [draw, fill=gray!10,rounded corners] {$\mathbf{u}_1^{0}$};
\end{tikzpicture}
\caption{A part of the full link of vertex $X^{-1}$.}
\label{j1xm1}
\end{figure}

From \cite{uncountably}, the full link of vertex $X^{-1}$ is isomorphic to the \hyperlink{sdbl}{spherical double} of $L$. Therefore there must be an edge between vertices $\mathbf{u}_1^0$ and $\mathbf{u}_2^{-1}$ in Figure \ref{j1xm1}. By translation, this means that in the full link of $x_1^{-1}\cdot X^{-1}$, the edge $x_1^{-1}\cdot\mathbf{u}_1^0$ is joined to the edge $x_1^{-1}\cdot\mathbf{u}_2^{-1}$. Therefore we deduce $\mu=x_1^{-1}$. This completes the square $\underline{\mathbf{s}}_1^{-1}$.

Inductive step: ($i>1$) Assume that the claim has already been proven for all $i$ up to and including $n\geqslant 1$. Consider the pyramid with apex $X^{-(n+2)}$. Applying (1.) gives the top half of square $\underline{\mathbf{s}}_1^{-(n+1)}$. Applying (2.) gives the vertices in layer 0 at the bottom of the face of the pyramid, and working downwards using the inductive hypothesis, along with a consideration of the full link of vertex $X^{-(n+1)}$ as above, completes square $\underline{\mathbf{s}}_1^{n+1}$.

\end{proof}

This completes the proof of Lemma \ref{funlemma} for $j=1$.

\medskip

\textbf{Inductive step:} ($1<j<m$) Suppose that Lemma \ref{funlemma} has already been proven for $n$ up to and including $j-1$. Then, using exactly the same method of proof as above, we can also deduce the general form of the square $\underline{\mathbf{s}}_{n+1}^i$. The only modification is that we use the inductive hypothesis to deduce what the line from $X^0$ to $X^{-i}$ looks like, on the appropriate corner of the pyramid (where two faces of the pyramid meet).

\medskip

\textbf{Final step:} ($j=m$) Once we have deduced what the squares $\underline{\mathbf{s}}_n^i$ for $1\leqslant n\leqslant m-1$ look like, we can do exactly the same for the squares $\underline{\mathbf{s}}_m^i$. This time the coefficients break the pattern because of the new factor in (1.), resulting in a different form for the square $\underline{\mathbf{s}}_m^i$ in Figure \ref{fig:bigfun}. Note that because $x_1^i\cdots x_m^i\cdot X^i=X^i$, the vertex coefficients in any of the squares of the fundamental domain in Figure \ref{fig:bigfun} are not necessarily unique.

\medskip

This completes the proof of Lemma \ref{funlemma}.

\bibliographystyle{acm}

\end{document}